\documentclass[11pt,a4paper]{article}
\usepackage{amsmath,amssymb,amsthm}
\usepackage{graphicx}
\usepackage{hyperref}
\usepackage{accents}
\usepackage{tikz}
\usetikzlibrary{
    calc,
    intersections,
    angles,
    quotes,
    arrows.meta,
    decorations.markings
}
\usepackage{tikz-cd}
\newcommand{\MG}{\underaccent{\longsim}{\mathcal{M}}} % Minimal Geometry
\newcommand{\longsim}{\mathrel{\scalebox{1.3}[0.8]{\ensuremath{\sim}}}}
\newcommand{\PerpMark}[5]{%
  \coordinate (PM) at ($(#1)!#3!(#2)$);
  \draw[#5]
    ($(PM)!#4!90:(#2)$) --
    ($(PM)!#4!-90:(#2)$);
}
\theoremstyle{definition}
\newtheorem{thm}{Theorem}[section]
\newtheorem{cor}[thm]{Corollary}
\newtheorem{rem}[thm]{Remark}

\title{A Hypotenuse-Angle Axiomatization of the Hilbert Plane}
\author{Roberto Volpe}
\date{September 25, 2026}

\begin{document}
\maketitle

\begin{abstract}
We show that, in $\MG^-$ --- $\MG$, the Hilbert plane, deprived of the
Side-Angle-Side axiom --- a hypotenuse-angle criterion for right
triangles, together with the existence of a perpendicular from an
external point to a line and a ray correspondence principle governing
the ordering of angles, suffices to reconstruct SAS in full 
--- and hence, as standard consequences, SSS
and SAA as well. This offers a more elementary alternative to Hilbert's
own choice of SAS as an axiom, confined throughout to right triangles and
provable without any continuity assumption. Several recent papers have
proposed related alternatives on comparable grounds; we compare our
approach to theirs in the concluding remarks. The result builds on, and
is best read alongside, two earlier papers of ours on related
reconstructions of SAS \cite{Volpe2026I,Volpe2026II}.
\end{abstract}

\section{Introduction}

Hilbert's choice to adopt SAS as an axiom \cite{Hilbert1950}, rather than
derive it, has recently been revisited on grounds of intuitiveness.
Euclid's own proof of the analogous result (\emph{Elements} I.4) proceeds
by superposition, a method never stated among his postulates or common
notions, and one that Hilbert avoided by positing the congruence criterion
directly.

Several recent contributions have proposed alternatives. H\"ahl and Peters
\cite{HaehlPeters2022} replace Hilbert's angle-transport and SAS axioms
(III4, III5) with three weaker-looking axioms, of which the central one,
(CT), directly postulates the existence of a triangle congruent to a given
one on a prescribed side of a segment --- in effect, a direct
axiomatization of Euclid's superposition. Johnson \cite{Johnson2025}
instead builds absolute geometry from two axioms concerning right angles
alone, from which full rigidity and SAS-equivalent results are recovered
without continuity. Edwards and Pambuccian \cite{EdwardsPambuccian2026},
working in Tarski's language of pure betweenness and equidistance, derive
Tarski's five-segment axiom (itself equivalent to SAS) from an axiom
system whose distinguishing feature is a rigorous formalization of
Euclid's Fourth Postulate --- again a statement confined to right
triangles, and again without continuity. A rather different tradition, of
which \cite{SJCR2015} is a comprehensive recent example, replaces triangle
congruence axioms altogether with an axiom asserting the existence of a
well-behaved set of line reflections, from which congruence is
\emph{defined} via composition of reflections.

Two earlier papers of ours \cite{Volpe2026I,Volpe2026II} work within
$\MG^-$, the axiom system obtained from $\MG$ --- itself known as the
Hilbert plane \cite{Hartshorne2000} --- by removing the SAS axiom.
They ask which of the classical triangle congruence criteria --- SSS,
ASA, SAA --- can serve as an adequate substitute for SAS, once paired
with auxiliary principles. Those principles are chosen so as to be
provable in $\MG$ itself: their addition to $\MG^-$ never strays outside
$\MG$.
A central role in every reconstruction is played by [\textbf{RCT}],
necessary in all three cases. It is Hilbert's own Theorem 13
\cite[Theorem~13]{Hilbert1950}, comparing two angles by inclusion of the
ray-pair of one within that of the other. The first of the two papers
above showed that the full order theory of angles follows from it alone
--- trichotomy, transitivity, addition and subtraction --- and that this
already suffices for the reasoning used throughout absolute geometry.
That paper went on to show that ASA, together with
[\textbf{RCT}], suffices on its own; SSS requires, in addition to
[\textbf{RCT}], the hypotenuse-angle criterion [\textbf{HA}] for right
triangles; SAA requires, in addition to [\textbf{RCT}], the Pons Asinorum
[\textbf{PA}]. The second paper then showed that [\textbf{PA}] is in fact
a theorem on the SAA branch, not an independent hypothesis.

The present note pursues a complementary strategy: rather than deriving
SAS from SSS as in \cite[Theorem~3.11]{Volpe2026I}, we aim to recover
SAS directly from [\textbf{RCT}] and two further principles confined to
right triangles, without any appeal to SSS. These are:

\begin{quote}
[\textbf{HA}] \emph{(Hypotenuse-angle.)} Two right triangles with
congruent hypotenuses and congruent acute angles, the right angle lying
opposite the acute angle in question, are congruent.
\end{quote}

\begin{quote}
[\textbf{PE}] \emph{(Existence of the perpendicular.)} Given a line $l$
and a point $P$ not on $l$, there exists a line through $P$ perpendicular
to $l$.
\end{quote}

\noindent Neither principle presupposes anything beyond its own statement:
[\textbf{HA}] says nothing about triangles in general, and nothing
presupposes that a right angle exists anywhere else in the plane; a
perpendicular is, by definition, a line forming a right angle, so
[\textbf{PE}] already carries with it the existence of at least one right
angle, with no further axiom needed --- and, since [\textbf{HA}] alone
already forces its uniqueness \cite[Theorem~3.8]{Volpe2026I}, [\textbf{PE}]
asserts existence only, as a genuinely independent primitive. In
particular, we never assume that all right angles are congruent to one
another: no counterpart of Euclid's Fourth Postulate is needed. It is
this combination --- one principle confined to right triangles, one
asserting the existence of a single right angle, and one governing how
angles compare --- that we show suffices to reconstruct SAS in full,
without any appeal to SSS.

\section{The main theorem}

We state the theorem below without repeating its proof in full: the
argument is that of \cite[Theorem~3.11]{Volpe2026I}, to which we refer
the reader for every detail. Two figures \ref{thm:main_f} and \ref{thm:main_f2} accompany the proof sketch that
follows, to help visualize the construction.

\begin{figure}
\centering
\begin{tikzpicture}
\coordinate (B) at (0,0);
\coordinate (C) at (5,0.5);
\coordinate (A) at (1.75,2.75);

\coordinate (E) at (8,0);
\coordinate (F) at (13,0.5);
\coordinate (D) at (9.75,2.75);

\draw[very thick] (A) -- (B) -- (C) -- cycle;
\draw[very thick] (D) -- (E) -- (F) -- cycle;

\fill (A) circle (1.5pt) node[above] {$A$};
\fill (B) circle (1.5pt) node[left] {$B$};
\fill (C) circle (1.5pt) node[right] {$C$};

\fill (D) circle (1.5pt) node[above] {$D$};
\fill (E) circle (1.5pt) node[left] {$E$};
\fill (F) circle (1.5pt) node[right] {$F$};
\coordinate (P) at ($(B)!(A)!(C)$);
\draw[dashed] (A) -- (P);

\fill (P) circle (1.5pt) node[below] {$P$};
\draw[densely dotted] (A) -- (P);
\coordinate (Q) at ($(E)!(D)!(F)$);
\draw[dashed] (D) -- (Q);

\fill (Q) circle (1.5pt) node[below] {$\quad Q$};
\draw[densely dotted] (D) -- (Q);
\coordinate (Pprime) at ($(E)!0.75!(Q)$);
\fill (Pprime) circle (1.5pt) node[below] {$\quad P'$};
\coordinate (I) at ($(D)!(F)!(Pprime)$);
\coordinate (L) at ($(D)!(E)!(Pprime)$);

\draw[densely dotted] (F) -- (I);
\draw[densely dotted] (E) -- (L);

\fill (I) circle (1.5pt) node[below] {$I\quad$};
\fill (L) circle (1.5pt) node[below] {$L\quad$};

\draw[densely dotted] (D) -- (L);

%\draw ($(A)!0.5!(B)+(-0.12,0.08)$) -- ($(A)!0.5!(B)+(0.12,-0.08)$);
%\draw ($(D)!0.5!(E)+(-0.12,0.08)$) -- ($(D)!0.5!(E)+(0.12,-0.08)$);
%   \PerpMark{A}{B}{x}{h}{style}
\PerpMark{A}{B}{0.5}{0.05}{thin}
\PerpMark{D}{E}{0.5}{0.05}{thin}
%
%\draw ($(A)!0.47!(C)$) ++(-0.055,-0.075) -- ++(0.11,0.15);
%\draw ($(A)!0.53!(C)$) ++(-0.055,-0.075) -- ++(0.11,0.15);
%
%\draw ($(D)!0.47!(F)$) ++(-0.055,-0.075) -- ++(0.11,0.15);
%\draw ($(D)!0.53!(F)$) ++(-0.055,-0.075) -- ++(0.11,0.15);
\PerpMark{A}{C}{0.49}{0.05}{thin}
\PerpMark{A}{C}{0.51}{0.05}{thin}
\PerpMark{D}{F}{0.49}{0.05}{thin}
\PerpMark{D}{F}{0.51}{0.05}{thin}
%%%%
\pic[draw, angle radius=0.35cm] {angle=B--A--C};
\pic[draw, angle radius=0.35cm] {angle=E--D--F};

\pic[draw, angle radius=0.65cm] {angle=P--A--C};
\pic[draw, angle radius=0.70cm] {angle=P--A--C};

\pic[draw, angle radius=0.65cm] {angle=Pprime--D--F};
\pic[draw, angle radius=0.70cm] {angle=Pprime--D--F};
\end{tikzpicture}
\caption{Theorem \ref{thm:main}-Proof schematic, case $B-P-C$}\label{thm:main_f}
\end{figure}
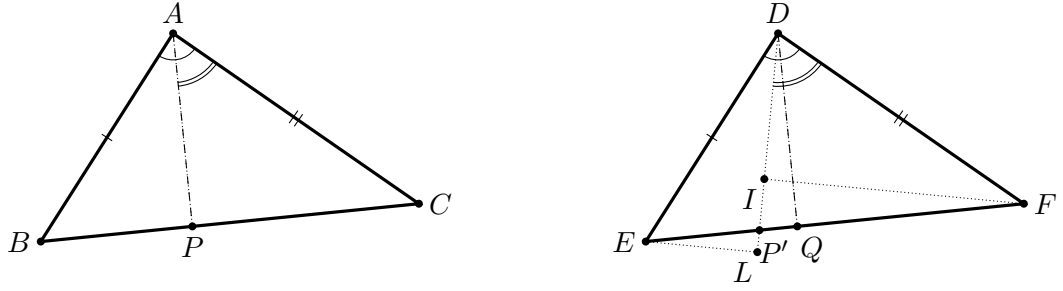
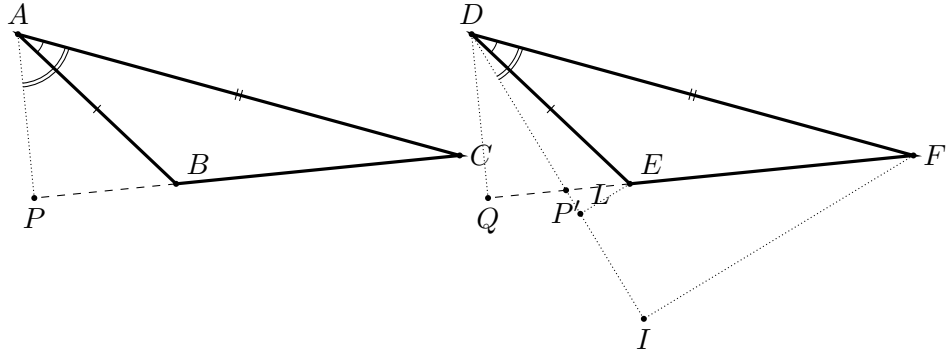
\begin{figure}
\centering
\begin{tikzpicture}[scale=0.75]
\coordinate (B) at (0,0);
\coordinate (C) at (5,0.5);
\coordinate (A) at (-2.789,2.636);

\coordinate (E) at (8,0);
\coordinate (F) at (13,0.5);
\coordinate (D) at (5.211,2.636);

\draw[very thick] (A) -- (B) -- (C) -- cycle;
\draw[very thick] (D) -- (E) -- (F) -- cycle;

\fill (A) circle (1.5pt) node[above] {$A$};
\fill (B) circle (1.5pt) node[above right] {$B$};
\fill (C) circle (1.5pt) node[right] {$C$};

\fill (D) circle (1.5pt) node[above] {$D$};
\fill (E) circle (1.5pt) node[above right] {$E$};
\fill (F) circle (1.5pt) node[right] {$F$};
\coordinate (P) at ($(B)!(A)!(C)$);
\draw[dashed] (B) -- (P);

\fill (P) circle (1.5pt) node[below] {$P$};
\draw[densely dotted] (A) -- (P);
\coordinate (Q) at ($(E)!(D)!(F)$);
\draw[dashed] (E) -- (Q);

\fill (Q) circle (1.5pt) node[below] {$Q$};
\draw[densely dotted] (D) -- (Q);
\coordinate (Pprime) at ($(Q)!0.55!(E)$);
\fill (Pprime) circle (1.5pt) node[below] {$P'$};
\coordinate (I) at ($(D)!(F)!(Pprime)$);
\coordinate (L) at ($(D)!(E)!(Pprime)$);

\draw[densely dotted] (F) -- (I);
\draw[densely dotted] (E) -- (L);

\fill (I) circle (1.5pt) node[below] {$I$};
\fill (L) circle (1.5pt) node[above right] {$L$};

\draw[densely dotted] (D) -- (I);
\PerpMark{A}{B}{0.5}{0.05}{thin}
\PerpMark{D}{E}{0.5}{0.05}{thin}
\PerpMark{A}{C}{0.495}{0.025}{thin}
\PerpMark{A}{C}{0.505}{0.025}{thin}
\PerpMark{D}{F}{0.495}{0.025}{thin}
\PerpMark{D}{F}{0.505}{0.025}{thin}
%\draw ($(A)!0.5!(B)+(-0.12,0.08)$) -- ($(A)!0.5!(B)+(0.12,-0.08)$);
%\draw ($(D)!0.5!(E)+(-0.12,0.08)$) -- ($(D)!0.5!(E)+(0.12,-0.08)$);
%
%\draw ($(A)!0.47!(C)$) ++(-0.055,-0.075) -- ++(0.11,0.15);
%\draw ($(A)!0.53!(C)$) ++(-0.055,-0.075) -- ++(0.11,0.15);
%
%\draw ($(D)!0.47!(F)$) ++(-0.055,-0.075) -- ++(0.11,0.15);
%\draw ($(D)!0.53!(F)$) ++(-0.055,-0.075) -- ++(0.11,0.15);

\pic[draw, angle radius=0.35cm] {angle=B--A--C};
\pic[draw, angle radius=0.35cm] {angle=E--D--F};

\pic[draw, angle radius=0.65cm] {angle=P--A--C};
\pic[draw, angle radius=0.70cm] {angle=P--A--C};

\pic[draw, angle radius=0.65cm] {angle=Pprime--D--F};
\pic[draw, angle radius=0.70cm] {angle=Pprime--D--F};
\end{tikzpicture}
\caption{Theorem \ref{thm:main}-Proof schematic, case $P-B-C$}\label{thm:main_f2}
\end{figure}

\begin{thm}\label{thm:main}
If [\textbf{HA}], [\textbf{RCT}], and [\textbf{PE}] hold in $\MG^-$, then
SAS holds.
\end{thm}

\begin{proof}[Proof sketch]
Let $\triangle ABC$ and $\triangle DEF$ be given with $AB\equiv DE$,
$AC\equiv DF$, and $\angle BAC\equiv\angle EDF$. By [\textbf{PE}], let $P$
be the foot of the perpendicular from $A$ to line $BC$, and $Q$ the foot
of the perpendicular from $D$ to line $EF$. From this point, the
argument proceeds exactly as in the proof of
\cite[Theorem~3.11]{Volpe2026I}: the case split on
$\angle PAC\equiv\angle QDF$ versus $\angle PAC\not\equiv\angle QDF$, the
construction of $P'$, and the closing argument via [\textbf{HA}] applied
twice, use only [\textbf{HA}], [\textbf{RCT}], and the uniqueness of the
perpendicular already noted above --- never SSS or SAA themselves. Every
step transfers verbatim once $P$ and $Q$ are secured by [\textbf{PE}],
in place of the branch-specific existence argument (Theorem~3.10 there)
on which \cite[Theorem~3.11]{Volpe2026I} otherwise relies.
\end{proof}

\begin{cor}\label{cor:consequences}
Under the hypotheses of Theorem~\ref{thm:main}, SSS and
SAA also hold, as standard consequences of SAS in
$\MG^-$ \cite{Greenberg1993}.
\end{cor}

\begin{rem}[Consistency and necessity of the hypotheses]\label{rem:meta}
The ordinary Euclidean plane satisfies [\textbf{RCT}], [\textbf{PE}], and
[\textbf{HA}] (indeed all of $\MG$), so the hypotheses of
Theorem~\ref{thm:main} are jointly consistent.

More is true: [\textbf{HA}] is not derivable from [\textbf{RCT}] and
[\textbf{PE}] alone. Recall from \cite{Volpe2026I} the model
$\mathbb{E}_H^2$, based on the counterexample Hilbert himself used
\cite[\S~11]{Hilbert1950}. The model $\mathbb{E}_H^2$ satisfies
[\textbf{RCT}] while failing both [\textbf{HA}] and SAS. We show that it
satisfies [\textbf{PE}] as well. Like [\textbf{RCT}], [\textbf{PE}] is a
statement about angles alone --- it asserts the existence of a right
angle, nothing about lengths --- and angle congruence in $\mathbb{E}_H^2$
is unchanged from the ordinary Euclidean plane; only segment length is
redefined there. So $\mathbb{E}_H^2$ satisfies [\textbf{PE}] by the same
token it satisfies [\textbf{RCT}]. By soundness,
\begin{equation*}
  [\textbf{RCT}],\ [\textbf{PE}]\ \not\vdash\ [\textbf{HA}],
\end{equation*}
and, independently, $[\textbf{RCT}],\ [\textbf{PE}]\not\vdash\text{SAS}$.
In this precise sense, [\textbf{HA}] is a genuinely necessary hypothesis
in Theorem~\ref{thm:main}, giving this reconstruction the same kind of
security, via the same model, as the two earlier reconstructions.
\end{rem}

\section{Concluding remarks}\label{sec:concl}

[\textbf{HA}] and [\textbf{PE}] compare favourably, on grounds of
intuitiveness, with the alternatives surveyed in the introduction. Unlike
H\"ahl and Peters' (CT), which postulates triangle transport directly
and so remains close in strength to SAS itself, [\textbf{HA}] is
confined to the narrow and visually immediate configuration of two right
triangles compared by hypotenuse and acute angle --- arguably closer to
ordinary geometric intuition than a general transport axiom. Unlike
Johnson's R1 and Edwards-Pambuccian's R13, both of which compare right
triangles via their \emph{legs}, [\textbf{HA}] compares hypotenuse and
angle, closer in character to Euclid's own Fourth Postulate.

The sense in which each of these alternatives claims to be ``more
intuitive'' than SAS varies, and we make no attempt to adjudicate among
them. H\"ahl and Peters ground intuitiveness in fidelity to Euclid's own
method of superposition \cite{HaehlPeters2022}; Edwards and Pambuccian
ground it in provability without continuity, without invoking the term
\cite{EdwardsPambuccian2026}; Johnson grounds it in the simplicity of the
axioms themselves and their proximity to Euclid's spirit
\cite{Johnson2025}; Specht et al.\ tie it more generally to conformity
with intuitive expectation throughout the development, without a direct
comparison to SAS \cite{SJCR2015}.

Measured against these criteria in turn, [\textbf{HA}] and [\textbf{PE}]
fare unevenly. On provability without continuity, the present
reconstruction succeeds unconditionally: the entire argument takes place
within $\MG^-$, with no continuity assumption anywhere. On simplicity and
right-angle confinement, it compares closely to Johnson's and
Edwards-Pambuccian's own criteria, at the cost of one further principle,
[\textbf{RCT}], beyond their two. On fidelity to Euclid's method of
superposition, a comparison would require entering into the historical
origins of the geometric concepts involved --- in particular, whether
comparing a hypotenuse and an angle already implicitly enacts a form of
superposition --- a question we are not equipped to adjudicate here.

\end{document}